\documentclass[a4paper,11pt]{amsart}

\usepackage{amssymb,amsmath,amsfonts}
\usepackage{amsfonts,amssymb}

\newtheorem{theorem}{Theorem}[section]

\newtheorem{statement}[theorem]{Statement}

\newtheorem*{Vedenisov's lemma}{Vedenisov's lemma}
\newtheorem{corollary}[theorem]{Corollary}

\theoremstyle{definition}
\newtheorem{definition}[theorem]{Definition}
\newtheorem{remark}[theorem]{Remark}
\newtheorem{construction}[theorem]{Construction}

\newcommand\ind{\operatorname{ind}}
\newcommand\Ind{\operatorname{Ind}}
\newcommand\Fr{\operatorname{Fr}}
\newcommand\st{\operatorname{st}}
\newcommand\diam{\operatorname{diam}}

\begin{document}

\title{On a Problem of Arkhangel'skii}

\author{I.~M.~Leibo}
\address{Moscow Center for Continuous Mathematical Education, Moscow, Russia}
\email{imleibo@mail.ru}

\begin{abstract}
The coincidence of the  $\Ind$ and $\dim$ dimensions for first countable paracompact $\sigma$-spaces is proved.
As a corollary, the equality $\Ind X= \dim X$  for every Nagata (that is, first countable
stratifiable) space $X$ is obtained. This gives a positive answer to A.~V.~Ar\-khan\-gel'\-skii's  question of whether
the dimensions $\ind$, $\Ind$, and $\dim$ coincide for first countable spaces with a countable network.
\end{abstract}


\markboth{I.~M.~Leibo}{On a Problem of Arkhangel'skii}

\maketitle

\section*{Introduction}

Throughout the paper, all spaces under consideration are assumed to be normal
$T_1$-spaces with finite dimension $\Ind$, unless otherwise stated. By virtue of Vedenisov's inequality
\cite{1}, they also have finite dimension~$\dim$.

An important role in this paper is played by the notion of a network introduced by Arkhangel'skii \cite{2},~\cite{3}.

A family $\mu$ of subsets of a space $X$ is called a \emph{network} if,
for each point $x\in X$ and any neighborhood $Ox$ of $x$, there exists an element $F\in \mu$ such
that $x\in F\subseteq Ox$.

  If $\mu$ is a  network  in a regular space $X$, then the closures of elements of $\mu$ also form a network.
This allows us to consider only \emph{closed networks}, i.e., networks consisting of closed sets. Therefore, in
what follows, we assume all networks under consideration to be closed, unless otherwise stated.

There exist three classical topological dimension functions, $\Ind$, $\dim$, and $\ind$ (see, e.g.,
\cite{1}); relationships between them for various classes of spaces is of special interest in dimension theory.
L.~A.~Tumarkin and W.~Hurewicz proved the equalities $\ind X=\dim X=\Ind X$ for a space $X$ with a countable base
\cite[Chap.~4, Sec.~8, Subsec.~2]{1}. Since every base is a network, it is quite natural to ask whether the
dimensions of a space with a countable network coincide; this question was posed by Arkhangel'skii in
\cite[Sec.~5]{2}. In the general case, the answer is negative. A space $X$ with a countable network for which
$\dim X=1$ but  $\Ind X=\ind X=2$ was constructed by Charalambous \cite{4} (note that the dimensions $\Ind$ and
$\ind$ are equal for any such space). When stating his coincidence problem, Arkhangel'skii did not exclude the
possibility that a counterexample may exist. Therefore, he simultaneously stated the dimension coincidence
problem for spaces with a countable network which satisfy certain additional conditions, the most interesting of
which are the presence of a topological group structure and the first axiom of countability (i.e., the existence
of a countable neighborhood base at every point). The question of whether $\ind G= \dim G= \Ind G$ for any
topological group $G$ with a countable network still remains open (for more than 50 years now).

In this paper we give a positive answer to Arkhangel'skii's question of whether $\Ind X=\dim X=\ind X$ for any
first countable space with a countable network.

In \cite{2} Arkhangel'skii introduced the class of spaces with a $\sigma$-discrete network. After that, is was
quite natural to consider the classes of spaces with a $\sigma$-locally finite network and with a
$\sigma$-closure preserving network. It turned out that these classes of spaces coincide; the spaces in these
classes are called \emph{$\sigma$-spaces}. Properties of $\sigma$-spaces and paracompact
$\sigma$-spaces were described in detail in Arkhangel'skii's paper~\cite{3}.

In this paper we also prove that $\dim X = \Ind X$ for any first countable paracompact
$\sigma$-space $X$.

By virtue of Arkhangel'skii's theorem \cite[Theorem~2.8]{2} any first countable paracompact $\sigma$-space $X$
is semimetrizable, i.e., there exists a semimetric on $X$ which generates the topology of~$X$.

For a first countable paracompact $\sigma$-space $X$ satisfying a certain additional condition, namely, the
existence of a semimetric continuous with respect to one variable and generating the topology of $X$, a
positive answer to Arkhangel'skii's question was given by the author in \cite{5}. For first countable
stratifiable spaces, the answer is positive as well; this result was announced by the author in 2019~\cite{6}.

\section{Basic Definitions, Known Facts, and Auxiliary Results}

The main new results of this paper are Theorems~\ref{t2.1} and~\ref{t2.3}  and Corollaries~\ref{c2.2}, \ref{c2.4},
\ref{c2.5}, and \ref{c2.6}. To obtain them, we need the following definitions, theorems, and statements.

\begin{definition}
\label{d1.1}
Given a normal space $Z$ of finite dimension $\Ind Z=n$,  we say that a closed subset $F$ of $Z$ and its open
neighborhood $OF$ \emph{determine the dimension} $\Ind Z$ of $Z$ if, for any open neighborhood $U$ of $F$ such that
$F \subseteq U \subseteq \overline{U}\subseteq OF$, we have $\Ind\operatorname{Fr}U \ge n-1$.
\end{definition}

\begin{definition}
\label{d1.2}
A family $\varphi=\{F_{\alpha}, OF_{\alpha}\}_{\alpha \in A}$ of closed  subsets $F_{\alpha}$
and their open neighborhoods $OF_{\alpha}$ in a normal space $Z$ is called an \emph{everywhere f-system} (an \emph{f-system})
if the family $\{OF_{\alpha}\}_{\alpha \in A}$ is $\sigma$-discrete and, given any (any closed) set $M \subseteq
Z$, there exists an index $\alpha_0 \in A$ for which $M \cap F_{\alpha_0}$ and $M\cap OF_{\alpha_0}$
determine the dimension $\Ind M$.
\end{definition}

We will write the system \{$F_{\alpha}$$\cap$M, $OF_{\alpha}$$\cap$M\} as $\varphi$$\cap$M. Throughout the paper,  given a subset $M$ of a topological space $X$, a system of subsets of $X$ of the form $\nu=\{F_\alpha: \alpha\in A\}$  we will write $\nu$$\cap$$M$ for \{$F_\alpha$$\cap$ M: $\alpha \in A\}$. Given two covers $\omega$ and $\lambda$, $\omega$$\succ$$\lambda$ means that $\omega$ is a refinement of $\lambda$. Considering a metric space, by $O_r$$x$ we denote the open ball of radius $r$ centered at $x$.

Note that if the space $Z$ in Definition~1.2 is a paracompact $\sigma$-space, then it suffices to require the family
$\{OF_{\alpha}\}_{{\alpha}{\in}{A}}$ to be $\sigma$-locally finite in~$Z$. This is implied by the following
statement.

\begin{statement}[{\cite[Corollary~1.11]{5}}]
\label{s1.3}
Let $X$ be a paracompact $\sigma$-space. Suppose that there exists
an f-system $\varphi=\{F_{\alpha}, OF_{\alpha}\}_{\alpha \in A}$ of closed subsets $F_{\alpha}$ of $X$ and
their open neighborhoods $OF_{\alpha}$ in $X$ such that the family $\{OF_{\alpha}\}$ is $\sigma$-locally finite
in $X$. Then $\dim X=\Ind X$.
\end{statement}

\begin{definition}[\cite{5}]
\label{d1.4}
A space having an everywhere f-system is called an \emph{everywhere f-space}, and
a space having an f-system is called an \emph{f-space}.
\end{definition}

 Using the notion of an f-space, we can state the following sufficient condition for the coincidence of the
dimensions $\Ind X$ and $\dim X$ of a normal space $X$.

\begin{theorem}[{\cite{5}, \cite[Theorem~1]{15}}]
\label{t1.5}
If $X$ is an f-space, then $\Ind X=\dim X$.
\end{theorem}

Note that Statement 1.3. is a complete analogue of Theorem 1.5. for the class of paracompact $\sigma$ spaces.

For a hereditarily normal space, Theorem~\ref{t1.5} can be refined as follows.

\begin{theorem}[{\cite{5}, \cite[Theorem~3]{15}}]
 \label{t1.6}
If $X$ is a hereditarily normal everywhere f-space, then
the following conditions are equivalent:
\begin{enumerate}
 \item[\rm(a)]
$\dim X\leq  n$\textup;
\item[\rm(b)]
$X=\bigcup_{i=1}^n{X_i}$, where $X_i$ is a $G_\delta$ set with $\dim {X_i}\leq 0$ for each $i=0,1,2,\dots,
n$.
\end{enumerate}
\end{theorem}

\begin{theorem}[{\cite{5}, \cite[Theorem~4]{15}}]
\label{t1.7}
For a  paracompact $\sigma$-space $X$ which is an everywhere f-space, the following conditions are equivalent:
\begin{enumerate}
\item[\rm(a)]
$\dim X\leq  n$\textup;
\item[\rm(b)]
$X$ is the image of a strongly zero-dimensional paracompact $\sigma$-space under a perfect ${\le}(n+1)$-to-$1$ map.
\end{enumerate}
\end{theorem}

\begin{definition}
\label{d1.8}
We say that a closed network $\gamma=\{F_{\alpha}:\alpha \in A\}$ in a space $Y$ is an \emph{S-network} if,
for any closed set $F\subseteq  Y$ and any open neighborhood $OF$ of $F$, there exists a family $\gamma(F)\subset
\gamma$ such that the union $\bigcup \gamma(F)$ of all elements of  $\gamma(F)$ is closed
in  $Y$  and $F\subseteq \bigcup \gamma(F)\subseteq OF$.
\end{definition}

  We refer to a space having a $\sigma$-closure-preserving S-network as an \emph{S-space}.
The author \cite{14} and, independently, Oka \cite{9} showed that, in the class of stratifiable spaces,
being an S-space  is sufficient for  having equal dimensions $\dim$ and~$\Ind$.

Recall that a $T_1$ normal paracompact $\sigma$-space $X$ is Hausdorff,
hereditarily paracompact, hereditarily normal, and perfectly normal; moreover, $X$ is a
hereditarily $\sigma$-space.

In the class of perfectly normal spaces, the  monotonicity theorem for $\Ind$ holds~\cite{1}; namely,  if $X$ is a
perfectly normal space and $D$ is any subset of $X$, then $\Ind D\leq \Ind X$.

In what follows, we need Vedenisov's lemma, which we state below in a form convenient for our purposes.

 \begin{Vedenisov's lemma}[{\cite{1}, \cite[Lemma~2.3.16]{7}}]
Let $X$ be a perfectly normal space, and let $A$ and $B$ be disjoint closed subsets of $X$.
Suppose given a  countable open cover $\{G_i, i=1,2,\dots\}$ of $X$ such that, for each $i=1,2,\dots$,
$\Ind\Fr G_i\leq n$ and either
$A\cap \overline{G_i} =\varnothing$ or $\overline{G_i}\cap B=\varnothing$.
Then there exists a  partition $L$ between $A$ and $B$ in $X$ such that $L \subseteq \bigl(\bigcup \{\Fr G_i:
i=1,2,\dots\}\bigr)$
and, therefore, $\Ind L\leq n$ \textup(by virtue of \v Cech's countable sum theorem for $\Ind$ in a
perfectly normal space~\cite[Theorem~6, Chap.~7, Sec.~3]{1}\textup).
\end{Vedenisov's lemma}

  Let $X$ be a paracompact $\sigma$-space, and let $\eta =\{F_\alpha: \alpha\in A\}$ be a system of closed
subsets of $X$. Suppose that
$\Ind F_\alpha\leq n$ for each $\alpha \in A$. Then $\Ind\bigl(\bigcup\{F_\alpha: \alpha\in A\}\bigr)\leq n$,
provided that the system $\eta$ is either countable, locally finite, or
$\sigma$-locally finite in $X$~\cite{1}. Nagata proved a sum theorem
for $\Ind$ in the class of stratifiable spaces for closure-preserving systems $\eta$; namely,
he proved that if  a system of $\eta =\{F_\alpha: \alpha\in A\}$ of closed sets  in a stratifiable space $X$ is
closure-preserving and $\Ind F_\alpha\leq n$ for each $\alpha\in A$,
then $\Ind\bigl(\bigcup\{F_\alpha: \alpha\in A\}\bigr)\leq n$.

 The following statement is an immediate corollary of Vedenisov's lemma and the sum theorem for~$\Ind$.

 \begin{theorem}[{\cite[Lemma~1.8]{5}}]
\label{t1.9.}
Let $\gamma =\{F_\alpha:\alpha\in A\}$ be a closed $\sigma$-discrete network in
a paracompact $\sigma$-space $X$. Suppose that, for each $\alpha\in A$ and any open neighborhood $UF_\alpha$ of
$F_\alpha$, there exists an open neighborhood
$OF_\alpha$ of $F_\alpha$ such that $F_\alpha \subseteq OF_\alpha\subseteq
\overline{OF_\alpha}\subseteq UF_\alpha$
and $\Ind\Fr OF_\alpha\leq n-1$. Then $\Ind X\leq n$.
\end{theorem}

  Clearly, if a system $\{O_\alpha\}$ of open sets is $\sigma$-locally finite in a space $X$ and
$F_\alpha = X\setminus O_\alpha$, then the system $\{F_\alpha\}$
of closed sets is $\sigma$-closure-preserving~\cite{8}.

S.~Oka proved the following theorem.

\begin{theorem}[{\cite[Lemma~3.2]{9}}]
\label{t1.10}
  Let $X$ be a paracompact $\sigma$-space, and let  $\mathfrak {F}= \bigcup\{\mathfrak F_i: i=1,2,\dots\}$, where
each $\mathfrak F_i$ is a closure-preserving system of closed sets in $X$.
Then there exists a metric space $M$ and a one-to-one continuous map such that $f(F)$ is closed in $M$
for each $F\in \mathfrak{F}$ and $f(\mathfrak{F}_i)$ is closure-preserving in  $M$ for each
$i=1,2, \dots$\,.
\end{theorem}

 Theorem~\ref{t1.10} implies the following statement.

\begin{statement}[{\cite[Corollary~1.10]{5}}]
\label{s1.11}
Let $X$ be a paracompact $\sigma$-space, and let
$\{O_\alpha: \alpha\in A\}$ be a locally finite system of open sets in $X$. Then
there exists a metric space $M_1$ and
a continuous one-to-one map $\psi\colon X\to M_1$ such that $\psi(O_\alpha)$
is open in $M_1$ for each  $\alpha \in A$. Note that if $\{O_\alpha: \alpha\in A\}$ covers the space $X$,
then $\{\psi(O_\alpha): \alpha\in A\}$ is an open cover of the metric space~$M_1$.
\end{statement}

 \begin{remark}
\label{r1.12}
Let $X$ be a paracompact $\sigma$-space, and let
$\{\lambda_i:  i=1,2,\dots\}$ be a sequence of open
locally finite covers of $X$. Then there exists a sequence
$\{\lambda'_i:  i=1,2,\dots\}$ of open
locally finite covers of $X$ such that $\lambda'_i\succ\lambda i$ and $\lambda'_{i+1}\succ\lambda'_i$ for all
$i=1,2,\dots$\,.

 Indeed, given two covers $\mu$ and $\nu$, let $\mu\wedge \nu$ denote the new cover consisting of pairwise
intersections of elements of $\mu$ and $\nu$. We set $\lambda'_1=\lambda_1$ and define, by induction,
$\lambda'_{i+1} = \lambda'_i\wedge \lambda_{i+1}$. Every new cover $\lambda'_i$ is a locally finite open
cover of $X$ being a refinement of $\lambda_i$.
\end{remark}

 Statement \ref{s1.11} implies the following assertion.

\begin{statement}
\label{s1.13}
Let $X$ be a paracompact $\sigma$-space, and let  $\{\lambda_i: i=1,2,\dots\}$ be a sequence of open
locally finite covers of $X$. In view of Remark~\ref{r1.12}, we can assume that
$\lambda_{i+1}\succ \lambda_i$. There exists a sequence $\{\omega^{\circ}_i: i=1,2,\dots\}$ of open locally
finite  covers of $X$ such that $\omega^{\circ}_i$ is a refinement of $\lambda_i$
\textup(that is, $\omega^{\circ}_i\succ\lambda_i$\textup) and  there exists a metric space $Z$
and a continuous one-to-one map $\psi\colon X\to Z$ for which $\psi(\omega^{\circ}_i)$ is an open locally
finite cover of $Z$. Moreover, the cover $\omega^{\circ}_{i+1}$ can be assumed
to be a refinement of
$\omega^{\circ}_i$  \textup(that is, $\omega^{\circ}_{i+1}\succ\omega^{\circ}_i$\textup).
\end{statement}

\begin{proof}
By virtue of Statement~\ref{s1.11}, for the open locally finite cover $\lambda_i$ of $X$,
there exists a metric space
$Z_i$ and a continuous one-to-one map $\psi_i\colon X\to Z_i$ such that
$\psi_i(\lambda_i)$ is an open cover of $Z_i$. Let $\vartheta_i$ be a locally finite open refinement
of $\psi_i$($\lambda_i$).  We set $\omega^{\circ}_i = \psi^{-1}_i (\vartheta_i)$. Consider the
 diagonal
$$
\psi\colon X\to Z\subseteq\prod\{Z_i:i=1,2,\dots \}
$$
of the maps $\psi_i$.
It follows from properties of diagonal maps that  $\psi(\omega^{\circ}_i)$ is an open locally
finite cover of $Z$. The map $\psi$ is the required continuous one-to-one map of $X$ onto the metric space
$Z$. The sequence $\{\omega^{\circ}_i: i=1,2,\dots\}$ of  open locally finite covers of $X$ is as required
(as shown above, we can assume that
$\omega^{\circ}_{i+1}\succ\omega^{\circ}_i$). This completes the proof of the statement.
\end{proof}

Arkhangel'skii proved the following theorem.

\begin{theorem}[{\cite[Sec.~6]{3}}]
\label{t1.14}
Let $X$ be a paracompact $\sigma$-space, and let $\gamma =\{F_{\alpha}: \alpha\in A\}$,
where $A=\bigcup_{i=1}^{\infty}A_{i}$, be a closed $\sigma$-discrete network  in $X$
such that the family $\gamma_i =\{F_{\alpha}: \alpha\in A_{i}\}$ is discrete in $X$ for
each $i=1,2, \dots$\,. Then there exists a continuous one-to-one map
$\varphi \colon X\to Y$ of $X$ onto a metric space $Y$ such that, for each $\alpha\in A$, the set
$\varphi(F_{\alpha})$ is closed in $Y$ and the system $\varphi(\gamma_i)$ is discrete in $Y$,
so that $\varphi(\gamma)$ is a closed $\sigma$-discrete network in the space~$Y$.
\end{theorem}

 In \cite{10} Reed introduced the  notion of a developable subspace $M$
of a first countable space $X$, which is important for what follows.

\begin{definition}[\cite{10}]
\label{d1.15}
A subspace  $M$ of a first countable space $X$ is said to be \emph{developable} in $X$ if there
exists a sequence $\{G_i: i=1,2,3,\dots \}$ of open covers of $X$ such that, for each point $m\in M$
and each open neighborhood $Om$ of $m$, there exists an $n\in \mathbb N$ for which $\st(m, G_n)\subseteq Om$, and for each
point $x\in X$, there exists a nonincreasing sequence $\{U_k x : k=1,2,3,\dots \}$ of its open neighborhoods
such that each $U_kx$ is an element of $G_k$,
i.e., $U_kx\in G_k$; moreover, this sequence forms a local base of $X$ at $x\in X$.
\end{definition}

 \begin{theorem}[{\cite[Theorem~2.5]{5}}]
\label{t1.16}
Let $X$ be a first countable paracompact $\sigma$-space. Suppose that $X$
has a closed $\sigma$-discrete network  $\gamma =\{F_{\alpha}: \alpha\in A\}$ such that
$A=\bigcup_{i=1}^{\infty}A_{i}$ and the system $\gamma_i=\{F_{\alpha}:\alpha\in A_{i}\}$ is discrete in $X$ for
each $i=1,2, \dots$\,. Then $X$ contains a metrizable subspace $M$ such that $M$ is dense in $X$,
$M \cap F_{\alpha}$  is dense in $F_{\alpha}$ for each $\alpha\in A$, and $M$
is  developable in $X$.
\end{theorem}

  It follows from the construction of $M$ in \cite[Theorem~2.5]{5} that $\dim M=0$.

 \begin{remark}
\label{r1.17}
Since the subspace $M$ in Theorem~\ref{t1.16} is  developable in $X$, it follows that there
exists a sequence  $\{G_i: i=1,2,3,\dots \}$ of  open covers of $X$ such that, for every point $m\in M$
and every open
neighborhood $Om$ of $m$ in $X$, we have  $\st(m,G_n)\subseteq Om$ for some $n\in \mathbb N$.
Let $\lambda_i$ be an open locally finite refinement of $G_i$ for $i=1,2, \dots$\,. In view of Remark~\ref{r1.12},
we may assume that
$\lambda_{i+1}$ is a refinement of $\lambda_i$. Thus, the assumptions of Statement~\ref{s1.13} hold;
hence there exists  a sequence $\{\omega^{\circ}_i: i=1,2,\dots\}$ of open locally finite covers of $X$
such that each $\omega^{\circ}_i$ is a refinement of $\lambda_i$
($\omega^{\circ}_i\succ\lambda_i$), $\omega^{\circ}_{i+1}\succ\omega^{\circ}_i$, and  there exists
a metric space $Z$ and a continuous one-to-one map $\psi\colon X\to Z$ for which
$\psi(\omega^{\circ}_i)$ is an open locally finite cover of the space $\psi(X)=Z$.
Clearly, for every point $m\in M$ and every
open neighborhood $Om$ of $m$ in $X$, there exists an $n\in \mathbb N$ such that $\st(m,\omega^{\circ}_n)
\subseteq Om$. It follows that the restriction of $\psi\colon X\to Z$ to $M$ is a homeomorphic embedding, i.e.,
the subspace $M$ of $X$ is  homeomorphic to the subspace $\psi(M)$ of~$Z$.
\end{remark}

 \begin{remark}
\label{r1.18}
Thus, we have

(1)  a  continuous surjective one-to-one
map $\psi\colon X\to Z$ of a first countable paracompact $\sigma$-space $X$ onto a metric space $Z$
(see Statement~\ref{s1.13} and Remark~\ref{r1.17});

(2) a continuous surjective one-to-one map $\varphi\colon X\to Y$ constructed in Arkhangel'skii's theorem
(Theorem~\ref{t1.14}).

Consider the continuous one-to-one map $f=\psi\triangle\varphi$, i.e., the
diagonal of the maps $\psi$ and $\varphi$; this is a map $X\to M_0\subseteq Z\times Y$. Let
$\gamma=\{F_{\alpha}: \alpha\in A\}$ be a network in $X$ such that $A=\bigcup_{i=1}^{\infty}A_{i}$ and
the system $\gamma_i=\{F_{\alpha}: \alpha\in A_{i}\}$ is a discrete system of closed sets in $X$
for each $i=1,2, \dots$; thus, $\gamma$ is a closed $\sigma$-discrete network in $X$.
It follows from properties of diagonal maps that

(i)  $f(\gamma)$ is a  closed $\sigma$-discrete network in $M_0$;

 (ii) the subspace $M$ of $X$ is homeomorphic to the subspace $f(M)$ of the
metric  space~$M_0$;

 (iii) $f(\omega^{\circ}_i)$ is an open locally finite cover of~$M_0$;

 (iv) since $M_0$ is a metric space, each open cover $f(\omega^{\circ}_i)$ of $M_0$ has an open locally
finite refinement $\vartheta_i$ such that  $\diam G<1/i$ for every $G\in \vartheta_i$. We set
$\omega_i=f^{-1}(\vartheta_i)$. Thus,  if $G\in f(\omega_i)$, then $\diam G<1/i$.
We have obtained a sequence
$\{\omega_i: i=1,2,\dots\}$ of open locally finite covers of $X$ such that
$$
  \diam f(G)<1/i\text{ for } G\in \omega_i\quad \text{and}\quad \omega_{i+1}\succ\omega_i. \eqno{(1.18)}
$$
\end{remark}

  \begin{definition}
\label{d1.19}
We refer to the continuous one-to-one map
$f\colon X\to M_0\subseteq Z\times Y$ with properties (i)--(iv) in Remark~\ref{r1.18} as a \emph{weak special
condensation} of the  first countable  paracompact $\sigma$-space $X$ onto a metric space.
\end{definition}

  Let $X$ be a first countable paracompact $\sigma$-space. Then $X$ is semimetrizable by
Arkhangel'skii's theorem \cite{2}. Moreover, according to Theorem~\ref{t1.16}, $X$ has a metrizable subspace $M$
such that $M$ is dense in $X$ and
$M \cap F_{\alpha}$  is dense in $F_{\alpha}$. We say that a semimetric $d(x,y)$ on $X$
is an \emph{Arkhangel'skii semimetric} if it generates the topology of  $X$ and the restriction of $d(x,y)$
to $M$ is a metric. A method for constructing such semimetrics was proposed by Arkhangel'skii in
\cite[Theorem~2.8]{2} (this is why we call them ``Arkhangel'skii semimetrics'').
There may exist many Arkhangel'skii semimetrics on $X$; one of them is given in \cite{5}.
In what follows, we use a particular Arkhangel'skii semimetric constructed by
Arkhangel'skii's method of~\cite{2}.

 \begin{construction}
\label{con1.20}
\textbf{Construction of an Arkhangel'skii's semimetric}.
Let $X$ be a first countable paracompact $\sigma$-space. Then $X$ has a closed $\sigma$-discrete network
$\gamma = \{F_{\alpha}: \alpha\in A\}$, where $A=\bigcup_{i=1}^{\infty}A_{i}$,
and a dense metrizable subspace
$M\subseteq X$ such that $M\cap F_{\alpha}$ is dense in $F_{\alpha}$ for each $\alpha\in A$. According to
Arkhangel'skii's method in \cite[Theorem~2.8]{2}, to construct a semimetric on $X$, we must
determine an appropriate countable base of neighborhoods of each $x\in X$.

If $x\in M$, then for $O_nx$ we take $\st(x,\omega_n)$, where $\omega_n$ is as in Remark \ref{r1.18},\,(iv).

Suppose that $x\in X\setminus M$. Let $\omega_n$ be the  open locally finite  cover of $X$ in Remark~\ref{r1.18}.
Since $X$ is first countable, it follows that the point $x$ has a countable base of open
neighborhoods  $U_kx$
such that $U_{k+1}x\subseteq U_kx$. Let $Vx$ be the intersection of all those elements of $\omega_n$ which
contain $x$ (the number of such elements is finite), i.e., $Vx=\bigcap \{G: G\in \omega_n, x\in G\}$,
and let $k$ be the least positive integer such that $k>n$ and $U_kx\subseteq Vx$.  We set $O_nx=U_kx$.

Thus, we have constructed the countable local base $\{O_nx:x\in X\}$ of $X$ at each point $x\in X$. It is this
base which we use in the proofs of the main Theorems~\ref{t2.1} and~\ref{t2.3}.

In \cite[Theorem~2.8]{2} the following semimetric $d_1(x,y)$ on a  first countable paracompact $\sigma$-space
$X$ was defined. We have $O_{n+1}x\subseteq O_nx$, x$\in$X, $n\in \mathbb N$.
We set
$$
M^k_x=\bigcup\{F^i_{\alpha}: i\leq k, \alpha\in A_i, x\notin F^i_{\alpha}\},
\qquad \widetilde O_kx= O_kx\setminus M^k_x
$$
for $k=1,2,\dots$ and
$$
N(x,y)=\max\{n: \widetilde O_nx\cap \widetilde O_ny\cap (x\cup y)\neq \varnothing\}.
$$
The semimetric $d_1(x,y)$ is  defined by $d_1(x,y)=1/N(x,y)$.

The Arkhangel'skii semimetric $d(x,y)$ which we need is defined  as follows:

(1) if $x\in X\setminus M$ or $y\in X\setminus M$, then $d(x,y)=d_1(x,y)$;

(2) if $x\in M$ and $y\in M$, then $d(x,y)=\rho_0(f(x),f(y))$, where $f\colon X\to M_0\subseteq Z\times Y$
is the weak special condensation of $X$ onto $M_0$ (see Remark~\ref{r1.18}) and $\rho_0$ is the metric on
$M_0$.

Thus, both metrics on $M\subseteq X$ and on $f(M)\subseteq M_0$ are induced by the same metric
on $M_0$. Therefore, (1.18) implies the following property, which plays an important role
in the proofs of Theorems~\ref{t2.1} and~\ref{t2.3}:
$$
\diam(O_nx\cap M)<1/n\quad \text{for }x\in X\setminus M. \eqno{(1.20)}
$$
It is easy to check in the same way as in \cite[Theorem~2.8]{2} that $d(x,y)$ is a semimetric on $X$
and  the semimetrics $d(x,y)$ and $d_1(x,y)$ generate the topology of~$X$.
\end{construction}

The proofs of our  main theorems also use the notion of a semicanonical space, which was
essentially introduced by Dugundji in the paper \cite{12}, in which he also described the basic properties
of semicanonical spaces. The term ``semicanonical'' was suggested by Cauty.

\begin{definition}
\label{d1.21}
Let $A$ be a closed subset of a space $X$. The pair $(X,A)$ is said to be \emph{semicanonical} if there
exists an open cover $\omega$ of the open set $X\setminus A$ such that each open neighborhood $OA$ of $A$
contains an open neighborhood $UA$ of $A$ with the property that any $G\in \omega$ intersecting
$UA$ is contained in~$OA$.

A space $X$ such that  the pair $(X,A)$ is  semicanonical for every closed subset $A$ of $X$ is called a
\emph{semicanonical space}.

We refer to the cover $\omega$ as a \emph{semicanonical cover} for the pair $(X,A)$.
\end{definition}

\begin{remark}[a property of the semicanonical cover $\omega$]
\label{r1.22}
Let us denote $X\setminus OA$ by $B$. If an element $G$ of $\omega$ intersects $B$,
then $G\cap UA=\varnothing$, i.e., $\overline{G}\cap A=\varnothing$. It follows from the
arbitrariness of the neighborhood $OA$ of $A$ in Definition~\ref{d1.21} that $\overline{G}\cap A=\varnothing$
for each $G\in \omega$.  If $\omega$ is locally finite
in $X\setminus A$, then, by Definition~\ref{d1.21}, the system $\{G: G\cap B\neq\varnothing\}$ is
locally finite in the entire space $X$. Therefore, $A\cap \overline{\st(B, {\omega})}=\varnothing$.
\end{remark}

Remark~\ref{r1.22} has the following corollary.

\begin{corollary}
\label{c1.23}
Given a  paracompact $\sigma$-space $X$, let $A$ and $B$ be  two disjoint  closed subsets of $X$, and let
$\omega$ be a semicanonical cover for the pair $(X,A)$. Suppose that  the
boundary of each $G\in \omega$ has  dimension  $\Ind\Fr G\leq n-2$
and $\omega$ is locally finite in $X\setminus A$. Then there exists a partition $L$
of  dimension $\Ind L\leq n-2$ between $A$ and~$B$.
\end{corollary}

  This corollary is a direct consequence of Remark~\ref{r1.22}: it follows from the inclusion
$L\subseteq \bigcup\{\Fr G: G\in \omega, G\cap B\neq\varnothing\}$ and the sum theorem for the $\Ind$ dimension
of the elements of a locally finite system of closed
sets in a paracompact $\sigma$-space.

 Let $X$ be a paracompact $\sigma$-space, and let $\gamma=\{F_{\alpha}: \alpha\in A\}$ be some closed
$\sigma$-discrete network such that $A=\bigcup_{i=1}^{\infty}A_{i}$ and the system
$\gamma_i=\{F_{\alpha}:\alpha\in A_{i}\}$  is discrete in $X$ for each $i=1,2, \dots$\,. We set
$F_i=\bigcup\{F_{\alpha}: \alpha\in A_{i}\}$.

In this notation we can formulate Definition 1.24, Statement 1.25 and Statement 1.26.

\begin{definition}
\label{d1.24}
We say that the paracompact $\sigma$-space $X$ is \emph{almost semicanonical}
if each pair $(X, F_{i})$ is semicanonical.
\end{definition}

Note that Definition 1.24 requires the existence of at least one closed $\sigma$-discrete network $\gamma$=\{$F_{\alpha}\}$  with the property specified in the definition, and this property is not at all required from each network in X.

 \begin{statement}
\label{s1.25}
Suppose that $F_{\alpha}\in \gamma_i$ and the pair $(X, F_{i})$ is almost
semicanonical. Then so is every pair $(X, F_{\alpha})$, $\alpha\in A_{i}$.
\end{statement}

\begin{proof}
 We prove the statement by specifying a particular semicanonical cover $\omega_{\alpha}$  for every
pair $(X, F_{\alpha})$, $\alpha\in A_i$.  Let
$\omega_i$  be an open semicanonical cover of $X\setminus F_i$  for the pair $(X, F_i)$. Since
the system $\{F_{\alpha}: \alpha\in A_{i}\}$ is discrete in $X$, it follows that there exist
neighborhoods $OF_{\alpha}$ of its elements $F_\alpha$ such that the system $\{OF_{\alpha}: \alpha\in A_{i}\}$ is
discrete in $X$; moreover, there exists a system of neighborhoods
$\{VF_{\alpha}: \alpha\in A_{i}\}$ of $F_\alpha$ such that $F_{\alpha} \subseteq VF_{\alpha} \subseteq
\overline{VF_{\alpha}}\subseteq OF_{\alpha}$ for each $\alpha\in A_i$.
For a fixed $\alpha\in A_i$, consider the system
$$
\omega^1_i=\{G^1:
G^1=G\setminus (\bigcup\{\overline{VF_{\beta}}:  \beta\in A_{i}, \beta\neq\alpha, G\in \omega_i\})\}
$$
of open sets.
The required open semicanonical cover of $X\setminus F_{\alpha}$ for the pair $(X, F_{\alpha})$ is
$$
\omega_{\alpha}=\omega^1_i \cup \{OF_{\beta}: \beta\in A_{i}, \beta\neq\alpha\}.
$$

Indeed, take any closed subset $\Phi$ of  $X$ such that $\Phi\cap F_{\alpha}=\varnothing$ and let $\Phi_1 =
\Phi\setminus \bigcup\{VF_{\beta}: \beta\in A_{i}, \beta\neq\alpha\}$. By
construction, $\Phi_1\cap F_i=\varnothing$. Let $UF_i$ be the open neighborhood of $F_i$ satisfying the
conditions in Definition~1.21
for $OF_i=X\setminus \Phi_1$  and $\omega_i$  (recall that $\omega_i$
is an open semicanonical cover of $X\setminus F_i$  for the pair $(X, F_i)$).
 We set $UF_{\alpha} = UF_i\cap OF_{\alpha}$. By construction $\omega_{\alpha}$ is an open semicanonical cover
of $X\setminus F_{\alpha}$
and the neighborhood   $UF_{\alpha}$ of $F_{\alpha}$ is as required in
Definition~\ref{d1.21}, provided that the closed $\Phi$, which we used to define it, is disjoint from $F_{\alpha}$.
\end{proof}

Let us immediately note that Statement 1.25, unlike Statement 1.26, is not used in the proof of the main theorems. Statement 1.25 is given to the article for completeness of description of almost semicanonical.

 The converse is also true.

 \begin{statement}
\label{s1.26}
Suppose that every pair $(X, F_{\alpha})$,
$\alpha\in A_{i}$, is semicanonical. Then so is every pair $(X, F_{i})$,
where $F_i=\bigcup\{F_{\alpha}: \alpha\in A_{i}\}$.
\end{statement}

 \begin{proof}
Since the system $\{F_{\alpha}:\alpha\in A_{i}\}$ is discrete in $X$, it follows that every $F_\alpha$
has a neighborhood $OF_{\alpha}$ such that the system $\{OF_{\alpha}: \alpha\in A_{i}\}$ is discrete in $X$.
There also exists a system of neighborhoods $\{VF_{\alpha}: \alpha\in A_{i}\}$ such that $F_{\alpha}
\subseteq VF_{\alpha} \subseteq  \overline{VF_{\alpha}}\subseteq OF_{\alpha}$. For each pair $(X,
F_{\alpha})$, $\alpha\in A_{i}$, let $\omega_{\alpha}$ be an open semicanonical cover of
$X\setminus F_{\alpha}$, and for each closed set $\Phi_{\alpha}$ disjoint from $F_{\alpha}$, let
$U_1F_{\alpha}$ be the open set in the definition of a semicanonical pair.
We set  $UF_{\alpha}=U_1F_{\alpha}\cap VF_{\alpha}$.  Let us construct
a semicanonical cover $\upsilon_i$  for the pair $(X, F_i)$. Take any open cover of the open set
$X\setminus  \bigcup\{\overline{VF_{\alpha}}: \alpha\in A_{i}\}$. We denote this cover by
$\vartheta_i$. Then,  for each $\alpha\in A_{i}$, we choose those elements of the cover
$\omega_{\alpha}$ which intersect $OF_{\alpha}$ and take their intersection with $OF_{\alpha}$. We obtain
an open cover of the set $OF_{\alpha}$$\backslash$$F_{\alpha}$. Let us denote it by $\omega^0_{\alpha}$; thus,
$\omega^0_{\alpha}=\{G\cap OF_{\alpha}: G\in \omega_{\alpha}\}$ for
$\alpha\in A_{i}$.  The cover $\upsilon_i=\vartheta_i\cup\{\omega^0_{\alpha}:
\alpha\in A_{i}\}$ is semicanonical for the pair $(X, F_i)$. Now, for any closed subset
$\Phi_i$ of $X$  disjoint from $F_i$, the  neighborhood required in
Definition~\ref{d1.21} is $UF_i=\bigcup\{UF_{\alpha}: \alpha\in A_{i}\}$. This completes the proof of the
statement.
\end{proof}

 \begin{theorem}[{\cite[Theorem~1.17]{5}}]
\label{t1.27}
Any almost semicanonical space $X$ is an S-space and has an f-system and an everywhere f-system.
\end{theorem}

 Using Theorems~\ref{t1.5}, \ref{t1.6}, and \ref{t1.7}, we can make conclusions about dimensions
of almost semicanonical spaces.

Recall that, by Definition~\ref{d1.24}, the space $X$ in Theorem~\ref{t1.27} is a paracompact $\sigma$-space.

The class of almost semicanonical spaces is fairly large. For example, it contains all metric spaces and their
images under closed maps, as well as  first countable  paracompact $\sigma$-spaces with a 1-continuous semimetric.
In what follows, we show that it also includes all first countable paracompact $\sigma$-spaces and, therefore,
Nagata spaces and first countable spaces with a countable network.

 However, it follows from  San On's results in \cite{11} that stratifiable spaces are not generally
almost semicanonical.

The following statement gives a useful and important example. The proofs of the main theorems rely on the method
of proof of this statement.

\begin{statement}[{see \cite[Lemma~2.1(p.2.13)]{12}}]
\label{s1.28}
Any metric space $X$ is semicanonical.
\end{statement}

 \begin{proof}
Let $A$ be a closed subset of a metric space $X$ with metric $\rho$. The distance from every
point $x\in X\setminus A$ to the closed set $A$
is positive: $\rho(x,A)>0$. Let $r(x)=\rho(x,A)/4$. Let us remind you that $O_{r(x)}$x=\{y$\in$X: $\rho$(x,y)$<$r(x)\}.
We set $\omega=\{O_{r(x)}x: x\in X\setminus A\}$. Let us show that the open cover $\omega$ of the open set
$X\setminus A$ is semicanonical for the pair $(X,A)$. Take  any closed subset $B$ of $X$ disjoint from $A$.
For every point $y\in A$, we take its neighborhood $O_{r(y)}y$, where $r(y)=\rho(y,B)/4$.
Let $OA=\bigcup\{O_{r(y)}y: y\in A\}$. It follows from
the definition of the distance from a point to a closed set and the triangle inequality in the metric space $X$
that if $OA\cap O_{r(x)}x\neq\varnothing$, then $B\cap O_{r(x)}x=\varnothing$. Therefore, the open cover
$\omega$ is semicanonical  for the pair $(X,A)$.
\end{proof}

Note that the open cover $\omega$ of the open set $X\setminus A$ has the following useful property:
given any point $x\in X\setminus A$, there exists a number $a>0$ such that, for any point $z$ in its
neighborhood $O_{r(x)}x$, we have $\rho(z,A)>a$, i.e.,
$$
\rho(O_{r(x)}x,A)\geq a. \eqno{(1.28)}
$$
This property (1.28) follows from the triangle inequality.

 \section{Main Theorems}

The main results of this paper are Theorems~\ref{t2.1} and \ref{t2.3} and  Corollaries~\ref{c2.2}, \ref{c2.4},
\ref{c2.5}, and~\ref{c2.6}.

First, we consider first countable spaces with a countable network and prove Theorem~\ref{t2.1} for these
spaces. Recall that every space with a countable network is hereditarily finally compact (that is, any
open cover of any its subspace has a countable subcover) and hereditarily paracompact; moreover,
for any space $X$ with a countable network, the relations $\dim X\leqslant \ind X=\Ind X$ hold~\cite{1}.

To prove Theorem~\ref{t2.1}, we show first that any first countable space $X$ with a closed countable network
is almost semicanonical and then that it is an f-space and an everywhere f-space.

 \begin{theorem}
\label{t2.1}
For every first countable space $X$ with a countable network, the following conditions are equivalent:
\begin{enumerate}
\item[(a)] $\dim X\leq  n$\textup;
\item[(b)] $\Ind X\leq  n$\textup;
\item[(c)] $X=\bigcup_{i=1}^{n+1}{X_i}$, where each $X_i$ is a $G_\delta$ set and  $\dim {X_i}\leq0$
for $i=1,2,\dots, n+1$\textup;
\item[(d)] $X$ is the image of a strongly zero-dimensional first countable  space with
a countable network under a perfect ${\le}(n+1)$-to-$1$ map; in other words, there exists a first countable
space $X_0$ with a countable network
and a closed map $f\colon X_0\to X$ such that $\dim X_0=0$ and $|f^{-1}(x)|\leq n+1$ for each $x\in X$.
\end{enumerate}
\end{theorem}

 \begin{proof} Given a first countable space  $X$ with a countable network, let us
construct  an everywhere f-system in the space $X$. After that, Theorem~\ref{t2.1} will follow from
Theorems~\ref{t1.5}, \ref{t1.6}, and \ref{t1.7}. To prove the existence of an f-system and an everywhere
f-system in $X$, it suffices to show that $X$ is almost
semicanonical, but  we deliberately do not use Theorem~\ref{t1.27} and construct an f-system and an everywhere
f-system straightforwardly.

 Let $\{F_i: i=1,2,\dots\}$ be a closed countable network of $X$. By Theorem~\ref{t1.16}
$X$ has a dense metric subspace $M\subseteq X$ such that  $M\cap F_i$  is dense in $F_i$ for
$i=1,2,\dots$\,.  Moreover, by Remark~\ref{r1.18} there exists a weak special
condensation (see Definition~\ref{d1.19})  $f\colon X\to M_0$  of $X$ onto a metric
space $M_0$ with the following properties:
\begin{itemize}
\item
$f(F_i)$ is closed in $M_0$ for each $i=1,2,\dots$;
\item
the subspace $M$ of $X$ is homeomorphic to the subspace $f(M)$ of the metric space $M_0$;
\item
there exists a sequence $\{\omega_i: i=1,2,\dots\}$ of open locally finite covers of $X$ such
that $\diam f(G)<1/i$ for every $G\in \omega_i$  and $\omega_{i+1}\succ\omega_i$;
\item
$X$ is semimetrizable, and it  admits an Arkhangel'skii semimetric $d(x,y)$ such that
if $x\in M$ and $y\in M$, then $d(x,y)=\rho_0(f(x), f(y))$, where $\rho_0$ is the metric on $M_0$, and if
$x\in X\setminus M$, then $\diam (O_nx\cap M)<1/n$ (see Construction~\ref{con1.20}).
\end{itemize}
We assume that the metrizable subspace $M$ of $X$ is endowed with the metric $\rho$ defined by
$\rho(x,y)=\rho_0(f(x),f(y))$ (for $x,y\in M$). Clearly, we have  $\diam(G\cap M)<1/i$
for $G\in \omega_i$.

Fix an $i\in \mathbb N$. By analogy with Statement~\ref{s1.28}, we construct  an
open cover of the open set $M_0\setminus f(F_i)$ in $M_0$ by taking the open neighborhood $Ox=O_{r(x)}x$,
where $r(x)=\rho_0(x,f(F_i))/4$, of each point
$x\in f(M)\setminus f(F_i)$ in the space $M_0$. For any metric space $Z$ with
metric $d$ and any set $A\subseteq Z$, the function $l(z)=d(z,A)$,  $z\in Z$,
is continuous on $Z$; therefore, since the set
$f(M)\cap (M_0\setminus f(F_i))$ is dense in $M_0\setminus f(F_i)$, it follows that
$$
\varpi_i=\{O_{r(x)}x: x\in (f(M)\setminus f(F_i),  r(x)=\rho_0(x,f(F_i))/4\}
$$
is a cover of the
set $M_0\setminus f(F_i)$ open in the
space $M_0$. Note that, for $x\in f(M)\setminus f(F_i)$, we have
$\rho_0(x,f(F_i)) = \rho_0(x,f(F_i\cap M))$ in $M_0$, because
$F_i\cap M$ is dense in~$F_i$.

Consider the open cover
$$
\nu_i=f^{-1}(\varpi_i) =
\{f^{-1}(O_{r(x)}x): x\in (f(M)\setminus f(F_i),
r(x)=\rho_0(x,f(F_i))/4\}
$$
of $X\setminus F_i$. By construction, the open cover $\nu_i\cap M$
of the open
subset $M\setminus F_i$ of $M$ is semicanonical for the pair $(M, F_i\cap M)$ ($F_i\cap M$ is closed in
$M$). Let us show that this cover $\nu_i$ of $X\setminus F_i$ is semicanonical for the pair $(X,
F_i)$.

Take any closed  subset $\widetilde \Phi$ of $X$ disjoint from $F_i$. The set $\widetilde\Phi$
has a neighborhood $O\widetilde \Phi$ such that $\overline{O\widetilde\Phi}\cap F_i=\varnothing$. Let us
construct  an open neighborhood $UF_i$ of $F_i$ satisfying the condition
in Definition~\ref{d1.21} that  if $G\in \nu_i$
and  $G\cap UF_i\neq\varnothing$, then $G\subseteq (X\setminus \overline{O{\widetilde\Phi}}$), i.e.,
$G\cap \overline{O\widetilde\Phi}=\varnothing$. Clearly, in this case, we also have $G\cap \widetilde\Phi
=\varnothing$. Let us denote the closed
set $\overline{O\widetilde\Phi}$ by $\Phi$. Then $\Phi$ is a canonical closed set and $\Phi\cap M$
is dense in~$\Phi$.

Recall that on the metrizable subspace $M$ of $X$ the metric $\rho(x,y)=\rho_0(f(x),f(y))$ is defined
(for $x,y\in M$). Let $y\in M$. We refer to the distance $\rho(y,F_i\cap M)$  in the metric of $M$
as the \emph{distance from the point
$y\in M$ to the closed subset  $F_i$ of $X$} and denote it by $\rho(y,F_i)$; thus,
$\rho(y,F_i)=\rho(y,F_i\cap M)$ and $\rho(y,\Phi)=\rho(y,\Phi\cap M)$ for $y\in M$.
Consider the function
$$
g_i(y)=\rho(y,F_i)/(\rho(y,F_i)+\rho(y,\Phi))
$$
on the metric subspace $M$ of $X$. It is defined  and continuous on $M$, because both functions $\rho(y,F_i)$ and
$\rho(y,\Phi)$ are continuous on $M$, since these are distances from points to closed sets in the metric space
$M$. The function $g_i(y)$ takes  $M$ to $I=[0,1]$; we have $g_i(y)=0$ if and only if $y\in (F_i\cap M)$ and
$g_i(y)=1$ if and only if $y\in (\Phi\cap M)$. It follows from the definition of the function $g_i(y)$
that $g_i(y)<\rho(y,F_i)$ if $y\notin\Phi\cap M$.

 The subset of $M$ on which $g_i(y)<1/6$ is open in $M$ (because $g_i$ is continuous on $M$), and  it contains
$F_i\cap M$. Let us denote this subset by $U_M(F_i\cap M)$. This is an open subset of $M$ containing
only points of $M$. Take any point $y\in U_M(F_i\cap M)$. We have $g_i(y)<1/6$, i.e., $\rho(y,F_i)
/(\rho(y,F_i)+\rho(y,\Phi))<1/6$, whence $\rho(y,\Phi)>5\rho(y,F_i)$. It follows from the definition of the
distance from a point to a closed set and the triangle inequality on the metric space $M$
that there exists a point $y_1\in (F_i\cap M)$ for which $\rho(y_1, \Phi) > 4\rho(y_1, y)$. Take the
neighborhood $O_{r({y_1})}y_1$ of $y_1$, where $r(y_1)=\rho(y_1,\Phi)/4$. We
have  $y\in O_{r({y_1})}y_1$.

 We extend $g_i$ to
a continuous function $\tilde g_i$ on $M\cup F_i$ by setting $\tilde g_i(y)=0$
for each point $y\in F_i\setminus M$ and $\tilde g_i(y)=g_i(y)$ for
$y\in M$. By definition, $\tilde g_i$ takes  $M\cup F_i$ to $I=[0,1]$.
Let us prove that this extension is continuous.

 According to \cite[Problem~288, Chap.~2, Sec.~4]{13},
to prove that the continuous function $g_i$ defined on the subspace $M$ of $M\cup F_i$ (which is dense in
$M\cup F_i$ by Theorem~\ref{t1.16}) has a continuous extension $\tilde g_i$ to
the whole set $M\cup F_i$, it suffices to check that $g_i$ can be continuously extended to
$M\cup\{y^{\ast}\}$, where $y^{\ast}$ is an arbitrary point in $F_i\setminus M$, i.e.,
$y^{\ast}\in F_i\setminus M$. Extending $g_i$ to $M\cup \{y^{\ast}\}$ by setting
$g_i(y^{\ast})=0$, we obtain a function $g_i^{\ast}(y)$. Let us show that it is continuous, for which
we must check its
continuity at the point~$y^{\ast}$.

 Take an $\varepsilon>0$. Let us show that the point $y^{\ast}$ has an open  neighborhood $Oy^{\ast}$
in $M\cup\{y^{\ast}\}$  such that $g_i(y)<\varepsilon$ for each $y\in Oy^{\ast}$.

 There exists an $n\in \mathbb N$ for which $1/n<\varepsilon$. We define an open neighborhood $O_n y^{\ast}$ of
 $y^{\ast}$ in $X$ as in Construction~\ref{con1.20}.
 By virtue of property~(1.20) we
have $\diam(O_n y^{\ast}\cap  M)<1/n$. Since $y^{\ast} \notin \Phi$, it follows that there exists an open
 neighborhood $O^2 y^{\ast}$ of $y^{\ast}$ in $X$ for which $O^2 y^{\ast}\cap \Phi=\varnothing$.
 The intersection $O y^{\ast} = O^2 y^{\ast}\cap O_n y^{\ast} \cap (M\cup \{y^{\ast}\})$ is a
neighborhood of $y^{\ast}$ in $M\cup  \{y^{\ast}\}$.
 For each point $y\in O y^{\ast}$, we have $\rho(y, F_i\cap  M) \leq 1/n <\varepsilon$. In the
case under consideration,  $y\notin \Phi\cap M$. Therefore,
  $g_i(y)<\rho(y,F_i)<\varepsilon$. This proves the continuity of $g_i$ at the point~$y^{\ast}$.

 Thus, the function $\tilde g_i\colon M\cup  F_i\to I$  is continuous on the subspace
$M\cup  F_i$. Consider the  set $\tilde g_i^{-1}[0,1/6)$, which is open in $M\cup  F_i$. We denote it by
$\widetilde U F_i$. This set is an open neighborhood of  $F_i$ in $M\cup  F_i$.  By construction
$F_i\cup U_M(F_i\cap M)=\widetilde UF_i$. Let $UF_i$ be an open set in $X$ for which $UF_i\cap M\cup  F_i =
\widetilde U F_i$.

 The open neighborhood  $UF_i$ of the closed set $F_i$ in $X$ is as required in Definition~\ref{d1.21};
namely, if $G\in \nu_i$ and $G\cap UF_i \neq \varnothing$, then
$G\subseteq X\setminus \overline{\widetilde O{\Phi}}$, i.e., $G\cap \overline{\widetilde O{\Phi}}=\varnothing$.
This follows from the properties of $U_M(F_i\cap M)$ mentioned above. Clearly,  we also have
$G\cap \widetilde\Phi=\varnothing$. Thus, the cover $\nu_i$ of the open set $X\setminus F_i$ is
semicanonical for the pair $(X,F_i)$. Such a  construction can be performed for each $i\in \mathbb N$; i.e., for
each  $i\in \mathbb N$, there exists an open cover $\nu_i$ of the open set $X\setminus F_i$ which is
semicanonical for the pair $(X,F_i)$.

 Now we construct  a system of closed sets and their neighborhoods in the space $X$ which determine the
dimension $\Ind X$. Suppose that $\Ind X=n$.

 Consider the open cover
$$
\nu_i=f^{-1}(\varpi_i)=\{f^{-1}(O_{r(x)}x):
x\in (f(M)\setminus f(F_i),   r(x)=\rho_0(x, f(F_i))/4\}
$$
of the open set $X\setminus F_i$. Let
$\zeta_i= \{G_{\beta}: \beta\in B_i\}$ be its open locally finite refinement,
and let $\{V_{\beta}: \beta\in B_i\}$ be a closed shrinking of $\zeta_i$ constructed in a standard way. Each $V_{\beta}$ is closed in X (Remark 1.22).
Consider the systems  $\mu_i=\{(V_{\beta}, G_{\beta}):
\beta\in B_i\}$ of closed sets and their neighborhoods for $i=1,2, \dots$\,. Suppose that $i\in \mathbb N$ and the
system $\mu_i$ does not determine the dimension~$\Ind X$.
 Then, for each $\beta\in B_i$, the closed set $V_{\beta}$ has a neighborhood $OV_{\beta}$
such that
$$
V_{\beta}\subseteq OV_{\beta}\subseteq \overline{OV_{\beta}}\subseteq G_{\beta}\quad\text{and}\quad
\Ind\Fr OV_{\beta} \leq n-2.
$$
Let $\Phi$ be any closed subset of $X$ disjoint from $F_i$.
By Corollary~\ref{c1.23} there exists a partition $L$ with $\Ind L\leq n-2$ between the sets $\Phi$
and $F_i$. Thus, if none of the systems $\mu_i$ determines~$\Ind X$, then, for every element $F_i$
of the network and any neighborhood $WF_i$ of $F_i$, there exists a partition $L$ with $\Ind L\leq n-2$
between the closed sets $F_i$ and $X\setminus WF_i$. According to Vedenisov's lemma,
given any two disjoint closed sets in $X$, there exists a partition $P$ with $\Ind P\leq n-2$ between them.
Therefore, $\Ind X\leq n-1$, while we have assumed that $\Ind X=n$. This contradiction proves
that each system $\mu_i$, $i=1,2, \dots$, determines the dimension $\Ind X$.

 The system $\mu_i$, $i=1,2, \dots$, has a drawback, which does not allow us to apply Statement~\ref{s1.3};
namely, it is not $\sigma$-locally finite in the space $X$. This is easy to fix. As mentioned above,
$f(F_i)$ is closed in $M_0$ for each network element $F_i$, $i=1,2,\dots$,
where $f$ is a weak special condensation (see Definition~\ref{d1.19})
of the given  space $X$ onto the metric space $M_0$. Consider the open  subset
$$
W_j(F_i)=\{x\in X: \rho_0(f(x),f(F_i))<1/j\}
$$
of $X$. Let $W_j=X\setminus \overline{W_j(F_i)}$, and let
$\mu_{ij}$ denote the set of all
elements of $\mu_i$ lying in $W_j$, i.e.,
$$
\mu_{ij} = \{(V_{\beta}, G_{\beta}): G_{\beta}\subseteq W_j, \beta\in B_i\}.
$$
Then the system $\{G_{\beta}:  G_{\beta}\subseteq W_j,
\beta\in B_i\}$ is locally finite in $X$. Since $\zeta_i$ is a refinement of $\nu_i$, it follows that each element
 $G_{\beta}\in \zeta_i$  is a subset of some element of $\nu_i$. In view of Property~(1.28), there
exists an $a>0$ for which $\rho_0(f(G_{\beta}),f(F_i))>a$. We have  $1/j<a$ for some $j\in \mathbb N$.
Therefore,
$G_{\beta}\subseteq W_j$, whence $\mu_i=\bigcup \{\mu_{ij}: j=1,2,\dots \}$.

 We have obtained $\sigma$-locally finite systems $\mu_{ij}$, $i,j=1,2,\dots$, of closed sets and their
neighborhoods in $X$ determining the dimension $\Ind X$. Let us show that the set
$\mu=\{\mu_{ij}: i,j=1,2,\dots \}$ is an f-system in $X$, i.e., that it  determines not only the dimension
$\Ind X$ but also the dimension $\Ind $ of
any closed subset of $X$; in other words, if $D\subseteq X$ and $D$ is closed in $X$, then the system
$D\cap \mu$ determines $\Ind D$ (note
that $\Ind D\leq \Ind X$ by the monotonicity theorem).

 To do this, it suffices to show that each closed  subspace $D$ of $X$ is almost semicanonical.
Let $F_i^D$  be a network element in $D$,  and let $\Phi_1$ be a subset of $D$ closed in $D$ (and hence in $X$)
and disjoint from  $F_i^D$.
There exists a network element $F_i$ in $X$ for which $D\cap F_i =F_i^D$. Clearly,
$F_i\cap \Phi_1=\varnothing$. Recall that the cover $\nu_i$ of $X\setminus F_i$ constructed above  is
semicanonical for the pair $(X, F_i)$. The set $F_i$ has an open  neighborhood $UF_i$  in $X$ satisfying the
conditions in Definition~\ref{d1.21}. We set $UF_i^D=UF_i\cap D$. The open (in $D$) cover $\nu_i\cap D$
of the subspace $D$ is semicanonical for the pair $(D,F_i^D)$ in $D$. Indeed, let $G_D\in \nu_i\cap D$,
i.e.,  $G_D=G\cap D$ for some $G\in \nu_i$, and let $G_D\cap UF_i^D \neq \varnothing$. Then $G\cap UF_i
\neq \varnothing$.
Therefore, $G\cap \Phi_1=\varnothing$ and hence $G_D\cap \Phi_1=\varnothing$. So each closed  subspace $D$ of $X$ is almost semicanonical. Thus the family $D\cap \mu$ determines the dimension~$\Ind D$.

 We have shown that the  $\sigma$-locally finite family $\mu=\{\mu_{ij}: i,j=1,2,\dots \}$ is an f-system
in the space $X$, i.e., this family  determines not only the dimension $\Ind X$ but also the dimension $\Ind$ of
any closed subset of $X$; therefore, $\dim X=\Ind X$ by Theorem~\ref{t1.5}. Thus, for any first countable
space $X$ with a countable network, we have
$$
\ind X=\dim X=\Ind X,
$$
because $\dim X\leq  \ind X\leq \Ind X$ for finally compact spaces~\cite[Sec.~8,
Lemma~2]{1}.

 Now let us show that the $\sigma$-locally finite family $\mu=\{\mu_{ij}: i,j=1,2,\dots \}$ is an
everywhere f-system in the space $X$, i.e., this family determines not only the dimension $\Ind X$  but
also the dimension $\Ind$ of any
subset of  $X$. Let us fix an $i\in \mathbb N$ and consider the family~$\mu_i$.

 Now we need Vedenisov's lemma. Let $D$ be any subset of $X$. Suppose that $\Ind D=k$  (note
that  $\Ind D \leq  \Ind X$ by the monotonicity theorem). Let $A_1$  and  $B_1$ be two disjoint closed sets
in $D$ determining the dimension $\Ind D$ (i.e., such that $\Ind L\geq k-1$ for any closed partition $L$
between $A_1$  and  $B_1$  in $D$). There are closed sets $A$ and $B$  in $X$
for which $D\cap A=A_1$ and $D\cap B=B_1$ (the sets $A$ and $B$ may intersect outside
the subspace $D$). Take any point $p\in D$.  It has an open
neighborhood $Op$ in $X$  such that either $\overline{Op}\cap B=\varnothing$ or
$\overline{Op}\cap A=\varnothing$.  If $\gamma=\{F_i: i=1,2,\dots\}$ is a countable network in $X$, then
there exists an $F_i\in \gamma$ such that $p\in F_i\subseteq Op$. We refer to such an element of
$\gamma$ as marked. All marked elements of the countable network $\gamma$ for $p\in D$
form a countable system of network elements covering $D$. Renumbering them, we obtain a countable
family $\gamma_0=\{F_j:  j=1,2,\dots\}$ of marked elements of $\gamma$
covering $D$ which form a network of the subspace $D$. Thus, $D\cap \gamma_0$ is a network of $D$,
$\gamma_0\subseteq \gamma$,
$ D\subseteq \cup  \gamma_0$, and each $F_j\in \gamma_0$ is disjoint from either $A$ or
$B$. Each network element $F_j\in \gamma_0$ has a neighborhood $WF_j$
satisfying the condition in Vedenisov's lemma, i.e., such that either $\overline{WF_j}\cap A=\varnothing$ or
$\overline{WF_j}\cap B=\varnothing$. Therefore,
 given each pair $(X, F_j)$, there exists a  cover $\nu_ j$ of $X\setminus F_j$ semicanonical for this pair.
 The set $F_j$ has an open neighborhood $UF_j$ in $X$ satisfying the condition in Definition~\ref{d1.21}
for the pair $(F_j, X\setminus WF_j)$ of  closed sets.
Let $X\setminus WF_j=\Phi_j$.  Clearly, $\Phi_j\cap F_j=\varnothing$. We set
$$
UF_j^D=UF_j\cap D,\quad F_j^D=F_j\cap D,\quad\text{and}\quad \Phi_j^D=\Phi_j\cap D.
$$

 An argument similar to that used in the construction of an f-system proves that
the open cover $\nu_j\cap D$ of the subspace $D$ is
semicanonical for the pair $(F_j^D, \Phi_j^D)$ in $D$, i.e., if $G\in \nu_j\cap D$ and
$G\cap UF_j^D \neq \varnothing$, then $G\cap \Phi_j^D=\varnothing$. If the systems $\mu_j\cap D$ do not
determine the dimension $\Ind D$, then, for each $j=1,2,\dots$, there is a partition $P_j$ of an appropriate
dimension $\Ind P_j\leq k-2$
between $F_j^D$ and $\Phi_j^D$ in $D$
(see Corollary~\ref{c1.23}).
    Thus, by Vedenisov's lemma, there is a partition $L$ of an appropriate dimension $\Ind L\leq k-2$ between
the sets $A_1$  and  $B_1$  in $D$. But in the case under consideration, any
closed partition $L$ between $A_1$  and  $B_1$  in $D$ has dimension $\Ind L\geq k-1$.  Therefore,
the family $D\cap \mu$ determines the dimension  $\Ind D$.  Thus, the $\sigma$-locally finite
family $\mu=\{\mu_{ij}: i,j=1,2,\dots \}$  is
an everywhere f-system in $X$, i.e., it determines not only the dimension $\Ind X$ but
also the dimension $\Ind$ of any  subset of $X$. The application of Theorems~\ref{t1.6} and~\ref{t1.7}
completes the proof of Theorem~\ref{t2.1}.
\end{proof}

 Theorem~\ref{t2.1} and assertion (d) of Theorem~\ref{t2.1} imply the following result.

  \begin{corollary}
\label{c2.2}
Let  $X$ be a  first countable space with a countable network. Then $\ind X=\dim X=\Ind X$
and $X$ is an S-space \textup(see Definition~\ref{d1.8}\textup).
\end{corollary}

   \begin{theorem}
\label{t2.3}
For any first countable paracompact $\sigma$-space $X$, the following conditions
are equivalent:
\begin{enumerate}
\item[(a)]		
$\dim X\leq  n$\textup;
\item[(b)]
$\Ind X\leq  n$\textup;
\item[(c)]
$X=\bigcup\{X_i : i=1,2,\dots, n+1\}$, where~$X_i$  is a $G_\delta$  set and  $\dim {X_i}\leq0$
for $i=1,2,\dots n+1$\textup;
\item[(d)]
the space $X$ is the image of a strongly zero-dimensional first countable $\sigma$-space under a
perfect ${\le}(n+1)$-to-$1$ map; in other words, there exists a first countable paracompact
$\sigma$-space $X_0$ and a closed map $f\colon X_0 \to X$ such that $\dim X_0=0$ and $|f^{-1}(x)|\leq n+1$
for each $x\in X$.
\end{enumerate}
\end{theorem}

   \begin{proof}
Let $X$ be a  first countable paracompact $\sigma$-space, and  let $\gamma=\{F_{\alpha}: \alpha\in A\}$
be a closed $\sigma$-network in $X$ such that $A=\bigcup_{i=1}^{\infty} A_{i}$ and
the system $\gamma_i=\{F_{\alpha}: \alpha\in A_{i}\}$ is discrete in $X$ for
each $i=1,2, \dots$\,. We set $F_i=\bigcup \{F_{\alpha}: \alpha\in A_{i}\}$. We
must prove that the space $X$ has an everywhere f-system. Then Theorem~\ref{t2.3} will follow from
Theorems~\ref{t1.5}, \ref{t1.6}, and \ref{t1.7}. To prove the existence of an  everywhere f-system in $X$, we
will show that the space $X$ is almost semicanonical, i.e.,
that  each pair $(X, F_{i})$ is semicanonical. After that, the existence of an f-system and an everywhere
f-system in $X$ will follow from Theorem~\ref{t1.27}. According to Statement~\ref{s1.26}, to prove that
each pair $(X, F_{i})$ is semicanonical, it suffices to show that
the pair $(X, F_{\alpha})$ is semicanonical for all
$\alpha\in A_{i}$ and $i=1,2,\dots$\,.

  The proof of Theorem~\ref{t2.3} relies on the same basic methods as that of Theorem~\ref{t2.1}.

  Recall that, by Theorem~\ref{t1.16}, $X$ has a metrizable subspace $M$ such that $M$ is dense in $X$,
$M\cap F_{\alpha}$ is dense in $F_{\alpha}$ for each $\alpha\in A$, and $M$
is  developable in~$X$.

   According to Remark~\ref{r1.18}, there exists a weak special condensation $f\colon X\to M_0$
 of the given space $X$ onto a metric space $M_0$ (see Definition~\ref{d1.19}), which has the following
properties:
\begin{itemize}
\item
$f(\gamma)$ is a  closed $\sigma$-discrete network in $M_0$ and $f(\gamma_i)$ is discrete in $M_0$
for each $i=1,2, \dots$;
\item
the subspace $M$ of $X$ is homeomorphic to the subspace $f(M)$ of the metric  space~$M_0$;
\item
there exists a sequence $\{\omega_i: i=1,2,\dots\}$ of open locally finite covers of $X$ such
that  $\diam(f(G))<1/i$ for each $G\in \omega_i$  and $\omega_{i+1} \succ \omega_i$;
\item
$X$ is semimetrizable, and it  admits an Arkhangel'skii semimetric $d(x,y)$ such that
if $x\in M$ and $y\in M$, then $d(x,y)=\rho_0(f(x), f(y))$,
where $\rho_0$ is the metric on $M_0$, and if
$x\in X\setminus M$, then $\diam (O_nx\cap M)<1/n$ (see Construction~\ref{con1.20}).
\end{itemize}

We fix an $i\in \mathbb N$ and an $\alpha\in A_i$. Let us show that the pair $(X, F_{\alpha})$
is semicanonical. By analogy with Statement~\ref{s1.28}, we construct an  open
cover of the open set $M_0\setminus f(F_{\alpha})$ in $M_0$ by taking
the open neighborhood $Ox=O_{r(x)}x$, where $r(x)=\rho_0(x,f(F_{\alpha}))/4$,
of each point $x\in fM\setminus f(F_{\alpha})$. Given
any metric space $Z$ with metric $d$, the function $l(z)=d(z,A)$,  $z\in Z$, is continuous
on $Z$; therefore, since the set
$M\cap (M_0\setminus f(F_{\alpha}))$ is dense in $M_0\setminus f(F_{\alpha})$, it follows that
$$
\varpi_{\alpha}=\{O_{r(x)}x: x\in f(M)\setminus f(F_{\alpha}),  r(x)=\rho_0(x,f(F_{\alpha}))/4\}
$$
is an open cover of the open set $M_0\setminus f(F_{\alpha})$. Note that, for every point
$x\in f(M)\setminus f(F_{\alpha})$, we
have $\rho_0(x,f(F_{\alpha})) = \rho_0(x,f(F_{\alpha}\cap M))$, because $F_i\cap M $
is dense in $F_{\alpha}$.

Consider the open cover
$$
\nu_{\alpha}=f^{-1}(\varpi_{\alpha}) =
\{f^{-1}(O_{r(x)}x): x\in (f(M)\setminus f(F_{\alpha}),  r(x)=\rho_0(x,f(F_{\alpha}))/4\}
$$
of $X\setminus F_{\alpha}$. By construction the open cover $\nu_{\alpha}\cap M$ of the open subspace
$M\setminus F_{\alpha}$ is semicanonical for the pair $(M, F_{\alpha}\cap M)$ (the set
$F_{\alpha}\cap M$ is closed in $M$). Let us show that the
cover $\nu_\alpha$ of $X\setminus F_{\alpha}$ is semicanonical for the pair $(X, F_{\alpha})$.

Take any closed subset $\widetilde\Phi$ of $X$ disjoint from $F_{\alpha}$. The set
$\widetilde \Phi$ has a neighborhood $O\widetilde \Phi$ such that
$\overline{O{\widetilde\Phi}}\cap F_{\alpha}=\varnothing$.
Let  $UF_{\alpha}$ be an open neighborhood of $F_\alpha$ in $X$ satisfying the condition
in Definition~\ref{d1.21}, namely, such that  $G\subseteq (X\setminus \overline{O{\widetilde\Phi}})$, i.e.,
$G\cap \overline{O{\widetilde\Phi}}=\varnothing$, for any element $G$ of $\nu_{\alpha}$
intersecting $ UF_{\alpha}$. Clearly, in this case, we have $G\cap \widetilde\Phi=\varnothing$. Let us denote
the closed set $\overline{O{\widetilde\Phi}}$ by $\Phi$. Then $\Phi$ is a regular closed set and
$\Phi\cap M$ is dense in~$\Phi$.

Recall that on the metrizable subspace $M$ of $X$ the metric $\rho(x,y)=\rho_0(f(x),f(y))$ is defined
(for $x,y\in M$). Let $y\in M$. We refer to the distance $\rho(y,F_\alpha\cap M)$  in the metric of $M$
as the \emph{distance from the point
$y\in M$ to the closed subset  $F_\alpha$ of $X$} and denote it by $\rho(y,F_\alpha)$; thus,
$\rho(y,F_\alpha)=\rho(y,F_\alpha\cap M)$ and $\rho(y,\Phi)=\rho(y,\Phi\cap M)$ for $y\in M$.
Consider the function
$$
g_\alpha(y)=\rho(y,F_\alpha)/(\rho(y,F_\alpha)+\rho(y,\Phi))
$$
on the metric subspace $M$ of $X$. It is defined  and continuous on $M$, because both functions
$\rho(y,F_\alpha)$ and
$\rho(y,\Phi)$ are continuous on $M$, since these are distances from points to closed sets in the metric space
$M$. The function $g_\alpha(y)$ takes  $M$ to $I=[0,1]$; we have $g_\alpha(y)=0$ if and only if
$y\in (F_\alpha\cap M)$ and $g_\alpha(y)=1$ if and only if $y\in (\Phi\cap M)$. It follows from the definition
of the function $g_\alpha(y)$ that $g_\alpha(y)<\rho(y,F_\alpha)$ if $y\notin\Phi\cap M$.

 The subset of $M$ on which $g_\alpha(y)<1/6$ is open in $M$ (because $g_\alpha$ is continuous on $M$),
and  it contains $F_\alpha\cap M$. Let us denote this subset by $U_M(F_\alpha\cap M)$.
This is an open subset of $M$ containing
only points of $M$. Take any point $y\in U_M(F_\alpha\cap M)$. We have $g_\alpha(y)<1/6$, i.e., $\rho(y,F_\alpha)
/(\rho(y,F_\alpha)+\rho(y,\Phi))<1/6$, whence $\rho(y,\Phi)>5\rho(y,F_\alpha)$.
It follows from the definition of the
distance from a point to a closed set and the triangle inequality on the metric space $M$
that there exists a point $y_1\in (F_\alpha\cap M)$ for which $\rho(y_1, \Phi) > 4\rho(y_1, y)$. Take the
neighborhood $O_{r({y_1})}y_1$ of $y_1$, where $r(y_1)=\rho(y_1,\Phi)/4$. We
have  $y\in O_{r({y_1})}y_1$.

 The function $g_\alpha$ is defined and continuous on $M$. We extend it to
a continuous function $\tilde g_\alpha$ on $M\cup F_\alpha$ by setting $\tilde g_\alpha(y)=0$
for each point $y\in F_\alpha\setminus M$ and $\tilde g_\alpha(y)=g_\alpha(y)$ for
$y\in M$. By definition, $\tilde g_\alpha$ takes  $M\cup F_\alpha$ to $I=[0,1]$.
It is continuous on $M\cup F_\alpha$; this is proved in the same way as the continuity of
the corresponding function in Theorem~\ref{t2.1} by using \cite[Problem~288, Chap.~2, Sec.~4]{13} and
the neighborhoods $O_nx$ with property~(1.20), which form a local base of $X$ at $x\in X\setminus M$
in Construction~\ref{con1.20}.

  The function $\tilde g_{\alpha}(y): M\cup  F_{\alpha}\to I$, where $I=[0,1]$, is continuous on
the subspace $M\cup  F_{\alpha}$. Consider the  set $\tilde g_\alpha^{-1}[0,1/6)$, which is open in
$M\cup  F_\alpha$. We denote it by
$\widetilde U F_\alpha$. This set is an open neighborhood of  $F_\alpha$ in $M\cup  F_\alpha$.  By construction,
$F_\alpha\cup U_M(F_\alpha\cap M)=\widetilde UF_\alpha$. Let $UF_\alpha$ be an open set in $X$ for which
$UF_\alpha\cap M\cup  F_\alpha =
\widetilde U F_i$.

 The open neighborhood  $UF_\alpha$ of the closed set $F_\alpha$ in $X$ is as required in Definition~\ref{d1.21};
namely, if $G\in \nu_\alpha$ and $G\cap UF_\alpha \neq \varnothing$, then
$G\subseteq X\setminus \overline{\widetilde O{\Phi}}$, i.e., $G\cap \overline{\widetilde O{\Phi}}=\varnothing$.
This follows from the properties of $U_M(F_\alpha\cap M)$ mentioned above. Clearly,  we also have
$G\cap \widetilde\Phi=\varnothing$. Thus, the cover $\nu_\alpha$ of the open set $X\setminus F_\alpha$ is
semicanonical for the pair $(X,F_\alpha)$.

   We have proved that every pair $(X, F_{\alpha})$, $\alpha\in A$,  is semicanonical in $X$.
Let $F_i=\bigcup \{F_{\alpha}: \alpha\in A_{i}\}$, $i=1,2,\dots$\,. Then, by virtue of
Statement~\ref{s1.26}, each  pair $(X, F_{i})$ is semicanonical in $X$. Thus,
$X$ is an almost semicanonical space, so that it has an everywhere f-system
and an f-system by Theorem~\ref{t1.27}.

   This completes the proof of Theorem~\ref{t2.3}.
\end{proof}

Theorem~\ref{t2.3}\,(d) implies the following assertion.

 \begin{corollary}
\label{c2.4}
Let  $X$  be a first countable paracompact $\sigma$-space. Then $\dim X=\Ind X$ and $X$ is an S-space.

\end{corollary}

Note: We proved in the proof of Theorem 2.3 that for any closed $\sigma$-discrete network $\gamma$=\{$F_{\alpha}$\} each pair (X, $F_{\alpha}$)  is semicanonical in the first countable paracompact $\sigma$-space $X$. From here we immediately obtain Corollary 2.5 in a trivial and obvious way (let us take any closed set F in X and add it to the network $\gamma$, then the pair (X, F) will be semicanonical in X, because $\gamma$$\bigcup$F is $\sigma$-discrete network in $X$).

\begin{corollary}
\label{c2.5}
Let $X$ be first countable paracompact $\sigma$-space. Then the space X is semicanonical.

\end{corollary}

  Recall that every stratifiable space is a paracompact $\sigma$-space \cite{8}. First countable stratifiable
spaces are called Nagata spaces. Theorem~\ref{t2.3} has the following corollary.

\begin{corollary}
\label{c2.6}
For any  Nagata space $X$, the following conditions are equivalent:
\begin{enumerate}
\item[(a)]
$\dim X\leq  n$\textup;
\item[(b)]
$\Ind X\leq  n$\textup;
\item[(c)]
$X=\bigcup\{X_i: i=1,2,\dots n+1\}$, where each $X_i$  is a $G_\delta$  set and  $\dim {X_i}\leq0$
for $i=1,2,\dots n+1$\textup;
\item[(d)]
$X$ is the image of a strongly zero-dimensional first countable paracompact $\sigma$-space under a perfect
${\le}(n+1)$-to-$1$ map; in other words, there exists a first countable paracompact $\sigma$-space $X_0$ and
a closed map $f\colon X_0 \to X$ such that $\dim X_0=0$ and $|f^{-1}(x)|\leq n+1$ for each $x\in X$.
\end{enumerate}
\end{corollary}

It is easy to show that the space $X_0$ in (d) can be assumed to be a Nagata space. Therefore, condition (d)
can be formulated as follows:

\begin{enumerate}
\item[(d)]
there exists a Nagata space $X_0$ with $\dim X_0=0$ and a surjective closed map $f\colon X_0 \to
X$ such that $|f^{-1}(x)|\leq n+1$ for $x\in X$.
\end{enumerate}

Corollary~\ref{c2.6}\,(d) implies the following assertion.

 \begin{corollary}
\label{c2.7}
Let  $X$ be a Nagata space. Then  $\dim X=\Ind X$ and $X$ is an
S-space.
\end{corollary}

\section{Open Problems}

 In conclusion, we recall several open questions.

As is known~\cite{8}, any Nagata space is an $M_1$ space (that is, a space with a $\sigma$-closure preserving
base).

\smallskip

{\bf 1.}\enspace Let $X$ be an $M_1$ space. Is it true that $\dim X=\Ind X$? A positive answer to
this question would give a positive answer
to the question of whether $\dim X=\Ind X$ for any stratifiable space~$X$.

\smallskip

{\bf 2 {\rm(Arkhangel'skii)}.}\enspace
Let $X$ be a topological group whose underlying space has a countable network. Is it true that $\dim X=\Ind
X=\ind X$?

\smallskip

{\bf 3.}\enspace
Let $X$ be a stratifiable topological group. Is it true that $\dim X=\Ind X$?

\smallskip

{\bf 4.}\enspace Let $X$ be a normal Moore  space (that is, a developable regular space).
Is it true that $\Ind X=\dim X$?

\section*{Acknowledgments}

It is my pleasure to express gratitude to Professor A.~V.~Arkhangel'skii for attention and to Professor
 O.~V.~Sipacheva for useful discussions.


\begin{thebibliography}{99}

\bibitem{1}
P. S. Aleksandrov, B. A. Pasynkov,
\textit{Introduction to Dimension Theory},
Nauka, Moscow,
1973
(in Russian).

\bibitem{2}
A.~V.~Arkhangel'skii,
``Mappings and spaces,''
Uspekhi Mat. Nauk \textbf{21}
(4(130)), 133--184 (1966). English transl.:
Russian Math. Surveys
\textbf{21}
(4), 115--162 (1966).

\bibitem{3}
A.~V.~Arkhangel'skii,
``Classes of topological groups,''
Uspekhi Mat. Nauk
\textbf{36} (3(219)),
127--146 (1981). English transl.:
Russian Math. Surveys
\textbf{36} (3),
151--174 (1981).

\bibitem{4}
M.~G.~Charalambous,
``Resolving a question of Arkhangel'ski\u i's,''
Fund. Math.
\textbf{192} (1),
67--76 (2006).

\bibitem{5}
I. M. Leibo,
``Equality of dimensions for some paracompact $\sigma$-spaces,''
Mat. Zametki
\textbf{113} (4),
499--516 (2023).
English transl.:
Math. Notes
\textbf{113} (4),
488--501 (2023).

\bibitem{6}
I. M. Leibo,
``On the equality of dimensions for Nagata spaces,''
Proc. Int. Conference ``Modern Problems of Geometry and Topology and Their Applications,'' Tashkent, 2019,
p.~141.

\bibitem{7}
R. Engelking,
\textit{Dimension Theory},
North-Holland, Amsterdam, 1978.

\bibitem{8}
Shou Lin, Ziqiu Yun,
\textit{Generalized Metric Spaces and Mappings},
Atlantis Press, Amsterdam, 2016.


\bibitem{9}
S. Oka,
``Dimension of stratifiable spaces,''
Trans. Amer. Math. Soc.
\textbf{275} (1),
231--243 (1983).

\bibitem{10}
G. M. Reed,
``Concerning first countable spaces,''
Fund. Math.
\textbf{74} (1),
161--169 (1972).

\bibitem{11}
S. San-ou,
``A note on $\Xi$-product,''
J. Math. Soc. Japan
\textbf{29} (2),
281--285 (1977).

\bibitem{12}
J. Dugundji,
``An extension of Tietze's theorem,''
Pacific J. Math.
\textbf{1} (3),
353--367 (1951).

\bibitem{13}
A. V. Arkhangel'skii, V. I. Ponomarev,
\textit{Fundamentals of General Topology: Problems and Exercises},
 Kluwer, Dordrecht, 1984.

\bibitem{14}
I. M. Leibo,
``On the dimension of some spaces,''
Dokl. Akad. Nauk SSSR
\textbf{262} (1),
26--29 (1982).

\bibitem{15}
I. M. Leibo,
``On closed images of metric spaces,''
Soviet Math. Dokl.
\textbf{16},
1292--1295 (1975).

\end{thebibliography}
\end{document}